\PassOptionsToPackage{dvipsnames}{xcolor}
\documentclass{lmcs}
\pdfoutput=1

\usepackage{lastpage}
\lmcsdoi{15}{1}{20}
\lmcsheading{}{\pageref{LastPage}}{}{}%
{Feb.~23,~2018}{Mar.~05,~2019}{}

\pdfoutput=1

\newcommand{\lmcsonly}[1]{#1}
\newcommand{\arxivonly}[1]{}
\newcommand{\lmcsorarxiv}[2]{#1}

\usepackage[utf8]{inputenc}

\usepackage{amsmath,amsthm,amssymb}

\definecolor{darkgreen}{rgb}{0,0.45,0}
\definecolor{darkred}{rgb}{0.75,0,0}
\definecolor{darkblue}{rgb}{0,0,0.6}

\lmcsonly{
  \usepackage{etoolbox}
  \patchcmd{\paragraph}{\normalfont}{{\normalfont \bfseries}}{}{}
}

\lmcsorarxiv{ 
  \theoremstyle{plain}
  \newenvironment{theorem}{\begin{thm}}{\end{thm}\ignorespacesafterend}
  \newenvironment{lemma}{\begin{lem}}{\end{lem}\ignorespacesafterend}
  \newenvironment{corollary}{\begin{cor}}{\end{cor}\ignorespacesafterend}
  \newenvironment{definition}{\begin{defi}}{\end{defi}\ignorespacesafterend}
  \newenvironment{proposition}{\begin{prop}}{\end{prop}\ignorespacesafterend}
  \newtheorem{problem}[thm]{Problem}
  \newenvironment{example}{\begin{exa}}{\end{exa}\ignorespacesafterend}
  \theoremstyle{definition}
  \newtheorem{constrInternal}[thm]{Construction}
  \theoremstyle{remark}
  \newenvironment{remark}{\begin{rem}}{\end{rem}\ignorespacesafterend}
}{ 
  \theoremstyle{plain}
  \newtheorem{theorem}{Theorem}
  \newtheorem{lemma}[theorem]{Lemma}
  
  \newtheorem{problem}[theorem]{Problem}
  \newtheorem{proposition}[theorem]{Proposition}
  \theoremstyle{definition}
  \newtheorem{definition}[theorem]{Definition}
  \newtheorem{constrInternal}[theorem]{Construction}
  \newtheorem{example}[theorem]{Example}
  \theoremstyle{remark}
  \newtheorem{remark}{Remark}
  \newtheorem*{remark*}{Remark}
}

\theoremstyle{thmC}
\newtheorem{corC}[thm]{Corollary}

\newenvironment{construction}[2][]
  {\begin{constrInternal}[{for Problem~\ref{#2}\ifthenelse{\isempty{#1}}{}{; #1}}]}
  {\qed\end{constrInternal}}

\usepackage[T1]{fontenc}

\arxivonly{
  \newcommand{\keywords}[1]{\newcommand{\thekeywords}{#1}}
}

\newcommand{\grants}[1]{\newcommand{\thegrants}{#1}}

\usepackage{microtype}

\title{Displayed categories}

\lmcsorarxiv{
  \author[B.~Ahrens]{Benedikt Ahrens}
  \address{School of Computer Science, University of Birmingham, United Kingdom}
  \email{b.ahrens@cs.bham.ac.uk}

  \author[P.~LeF.~Lumsdaine]{Peter LeFanu Lumsdaine}
  \address{Department of Mathematics, Stockholm University, Sweden}
  \email{p.l.lumsdaine@math.su.se}
}{
  \author{Benedikt Ahrens \\ \normalsize \nolinkurl{b.ahrens@cs.bham.ac.uk}
    \and Peter LeFanu Lumsdaine \\ \normalsize \nolinkurl{p.l.lumsdaine@math.su.se}}
}


\grants{This
    material is based upon work supported by the National Science
    Foundation under agreement No.~DMS-1128155 and CMU 1150129-338510.
    This material is based upon work supported by the Air Force Office of 
    Scientific Research under
    award number FA9550-17-1-0363.
    During the preparation of this article the authors were also partly supported by
    the CoqHoTT ERC Grant 637339
    and the Swedish Research Council (VR) Grant 2015-03835 \emph{Constructive and category-theoretic foundations of mathematics},
    and benefited from a research visit funded by the EUTypes COST Action CA 15123.}

\usepackage[none]{optional}

\usepackage[pdf,all,2cell]{xy}\SelectTips{cm}{10}\SilentMatrices\UseAllTwocells
\usepackage{xspace}

\usepackage{xifthen}

\usepackage{xcolor}

\usepackage{mathtools}
\mathtoolsset{mathic=true}

\newcommand{\C}{{\mathcal{C}}}
\newcommand{\D}{\mathcal{D}}
\newcommand{\E}{\mathcal{E}}
\newcommand{\J}{\mathcal{J}}
\newcommand{\T}{{\mathcal{T}}}

\newcommand{\f}{\bar{f}}
\newcommand{\g}{\bar{g}}
\newcommand{\h}{\bar{h}}

\newcommand{\constfont}[1]{\ensuremath{ \mathsf{#1} }}

\newcommand{\catname}[1]{\ensuremath{\constfont{#1}}\xspace}
\newcommand{\set}{\catname{Set}}
\newcommand{\type}{\catname{Type}}
\newcommand{\Grp}{\catname{Grp}}
\newcommand{\dGrp}{\Grp}
\newcommand{\Top}{\catname{Top}}
\newcommand{\dTop}{\Top}
\newcommand{\EndAlg}[2][]{\ensuremath{\text{\({#2}\)-\(\catname{EndAlg}\)}}}
\newcommand{\MonadAlg}[2][]{\ensuremath{\text{\({#2}\)-\(\catname{MonAlg}\)}}}
\newcommand{\Alg}[2][]{\ensuremath{\text{\({#2}\)-\(\catname{Alg}\)}}}
\newcommand{\dAlg}[2][]{\Alg[1]{#2}}

\newcommand{\psh}[1]{\widehat{#1}}

\newcommand{\dconst}[2][]{\mathop{\constfont{const}_{#1}}{#2}}

\newcommand{\Ty}[1][]{\constfont{Ty}\ifthenelse{\isempty{#1}}{}{(#1)}}
\newcommand{\compext}[2]{\ensuremath{{#1}.{#2}}}

\usepackage{pigpen}
\newcommand{\pb}{\ar@{}[dr]|<<{\text{\pigpenfont A}}}

\renewcommand{\hom}{\constfont{hom}}
\newcommand{\mor}[3][]{\hom^{#1}(#2,#3)}
\newcommand{\dmor}[4][]{\ensuremath{\hom^{#1}_{#3}({#2},{#4})}}
\newcommand{\dto}[3]{\ensuremath{{#1} \to_{#2} {#3}}}
\newcommand{\dob}[2]{\ensuremath{{#1}_{#2}}}
\newcommand{\compose}[2]{\ensuremath{{#1}\cdot{#2}}}
\newcommand{\total}[2][]{\ensuremath{\textstyle \int_{#1}{#2}}}
\newcommand{\forget}[1][]{\pi_1^{#1}}
\newcommand{\fiber}[2]{\ensuremath{{#1}_{#2}}}
\newcommand{\arr}[1]{#1^{\rightarrow}}
\newcommand{\darr}[1]{\arr{#1}}
\newcommand{\dom}{\constfont{dom}}
\newcommand{\cod}{\constfont{cod}}

\newcommand{\dreind}[2]{{#1}^*{#2}}

\newcommand{\slice}[2][-]{{#2}/{#1}}
\newcommand{\coslice}[2][-]{{#1}\backslash{#2}}

\newcommand{\Sigdisp}[2]{\sum_{#1} #2}

\newcommand{\idpath}[1][]{1_{#1}}
\newcommand{\idtoiso}{\ensuremath{\mathsf{idtoiso}}\xspace}
\newcommand{\iso}{\cong}
\newcommand{\Iso}[3][]{\constfont{iso}_{#1}({#2},{#3})} 
\newcommand{\diso}[3]{\ensuremath{ {#1} \iso_{#2} {#3} }}
\newcommand{\dIso}[4][]{\constfont{iso}^{#1}_{#3}({#2},{#4})}
\newcommand{\inv}[1]{{#1}^{-1}}

\newcommand{\defeq}{\coloneqq}

\newcommand{\depeq}[1][*]{=_{#1}}
\newcommand{\UniMath}{\href{https://github.com/UniMath/UniMath}{\nolinkurl{UniMath}}\xspace}
\newcommand{\coqdocbasebaseurl}{https://unimath.github.io/doc/UniMath/4dd5c17}
\newcommand{\coqdocbaseurl}{\coqdocbasebaseurl /UniMath.CategoryTheory.DisplayedCats.}
\newcommand{\urlhash}{\#}
\newcommand{\coqdocurl}[2]{\coqdocbaseurl #1.html\urlhash #2}
\newcommand{\nolinkcoqident}[1]{\nolinkurl{#1}} 
\makeatletter
\newcommand{\coqident}{\begingroup\@makeother\#\@coqident}
\newcommand{\@coqident}[3][]{%
  \ifthenelse{\isempty{#2}}%
  {\nolinkcoqident{#3}}%
  {\ifthenelse{\isempty{#1}}%
  {\href{\coqdocurl{#2}{#3}}{\nolinkcoqident{#3}}}%
  {\href{\coqdocurl{#2}{#3}}{\nolinkcoqident{#1}}}}%
\endgroup}

\newcommand{\gittagbaseurl}{https://github.com/UniMath/TypeTheory/tree/}
\newcommand{\gittagurl}[1]{\gittagbaseurl #1}
\newcommand{\gittag}{\begingroup\@makeother\#\@gittag}
\newcommand{\@gittag}[1]{\href{\gittagurl{#1}}{\nolinkurl{#1}}\endgroup}

\newcounter{saveenumi}
\newcommand{\saveitem}{\setcounter{saveenumi}{\value{enumi}}}
\newcommand{\restoreitem}{\setcounter{enumi}{\value{saveenumi}}}

\makeatother

\newcommand{\plan}[1]{}
\newcommand{\BA}[1]{}
\newcommand{\PLL}[1]{}
\opt{draft}{
\renewcommand{\plan}[1]{\textcolor{blue}{#1}}
\renewcommand{\BA}[1]{\textcolor{orange}{BA: #1}}
\renewcommand{\PLL}[1]{\textcolor{purple}{PLL: #1}}
}

\begin{document}

\begin{abstract}
We introduce and develop the notion of \emph{displayed categories}.

A displayed category over a category \(\C\) is equivalent to ‘a category \(\D\) and functor \(F : \D \to \C\)’, but instead of having a single collection of ‘objects of \(\D\)’ with a map to the objects of \(\C\), the objects are given as a family indexed by objects of \(\C\), and similarly for the morphisms.
This encapsulates a common way of building categories in practice, by starting with an existing category and adding extra data/properties to the objects and morphisms.

The interest of this seemingly trivial reformulation is that various properties of functors are more naturally defined as properties of the corresponding displayed categories.
Grothendieck fibrations, for example, when defined as certain functors, use equality on objects in their definition.
When defined instead as certain displayed categories, no reference to equality on objects is required.
Moreover, almost all examples of fibrations in nature are, in fact, categories whose standard construction can be seen as going via displayed categories.

We therefore propose displayed categories as a basis for the development of fibrations in the type-theoretic setting, and similarly for various other notions whose classical definitions involve equality on objects.

Besides giving a conceptual clarification of such issues, displayed categories also provide a powerful tool in computer formalisation, unifying and abstracting common constructions and proof techniques of category theory, and enabling modular reasoning about categories of multi-component structures.
As such, most of the material of this article has been formalised in Coq over the \UniMath library, with the aim of providing a practical library for use in further developments.

\arxivonly{
}
\end{abstract}

\maketitle

\clearpage

\opt{draft}{\tableofcontents}

\section{Introduction}

It is often said that reference to equality of objects of categories is in general both undesirable and unnecessary.

There are some topics, however, whose development does appear to require it.
One example often given is the definition of \emph{(Grothendieck) fibrations} (and their relatives): functors \(p : \D \to \C\) equipped with a lifting property providing (among other things) an object \(d\) of \(\D\) such that \(pd\) is equal to a previously given object \(c\) of \(\C\).
A similar example is the property of \emph{creating limits}; see \cite[Remark~5.3.7]{leinster} for an explicit discussion of this example.%
\footnote{Both of these definitions have analogues in which the equality is weakened to isomorphism; but the strict versions have nonetheless remained in more general currency.}

In examples of fibrations (or creation of limits), however, one virtually never has cause to speak explicitly of equality of objects; and equally in their basic general theory.

How is this avoidance achieved?
In the general development, equality occurs only within the notion of ‘objects of \(\D\) over \(c\)’, for objects \(c\) of \(\C\).
And in examples, there is almost always an obvious alternative notion of ‘object \(\D\) over \(c\)’, trivially equivalent to ‘objects of \(\D\) whose projection is equal to \(c\)’, but expressible without mentioning equality of objects.

Specifically, objects of \(\D\) typically consist of objects of \(\C\) equipped with extra data, structure, or properties; ‘an object of \(\D\) over \(c\)’ is then understood to mean ‘a choice of the extra data for \(c\)’.
For instance, in showing that the forgetful functor \(\Top \to \set\) creates limits, one doesn’t construct a space and then note that its underlying set is equal to the desired one; one simply constructs a suitable topology on that set.

The notion of \emph{displayed categories} makes this explicit.
A displayed category over \(\C\) consists of a family of types \(\D_c\) (of ‘objects over \(c\)’), indexed by objects \(c\) of \(\C\), and similarly sets of morphisms indexed by morphisms of \(\C\), along with suitable composition and identity operations to ensure that the total collections of objects and morphisms form a category (with a projection functor to \(\C\)).
This is entirely equivalent to the data of a category with a functor to \(\C\), just as ‘a family of sets indexed by \(X\)’ is equivalent to ‘a set with a function to \(X\)’.

If fibrationhood (or creating limits, etc.)\ is now defined not as a property of a functor but instead as a property of a displayed category, no mention of equality of objects is required.
Equality of objects is used only for turning an arbitrary functor into a displayed category; but this is rarely needed in practice, since most natural examples of fibrations, creation of limits, and so on already arise from displayed categories.
For instance, the standard definition of the category \(\Top\) can be read as the total category of a displayed category over \(\set\), whose objects over a set \(X\) are topologies on \(X\).

We therefore propose that displayed categories should be taken as a basis for the development of fibrations, creation of limits, and similar notions, in particular in the type-theoretic setting, where dealing with equality on objects is more practically problematic than in classical foundations.

We do not believe we are introducing something mathematically novel here; we are simply making explicit an aspect of how mathematicians already deal with certain kinds of examples in practice.
The payoffs, however, are twofold.

Firstly, since this concept has been previously un-articulated, it has not been consistently appreciated that it resolves the ‘problematic’ issue of fibrations (and various other notions) apparently requiring use of equality on objects.
Besides providing conceptual clarification, this should help in future work with disentangling which constructions genuinely \emph{do} require use of equality on objects, and hence may require extra work or assumptions to develop in type-theoretic settings.

Secondly, by making this common informal technique precise, we make it available for use in computer formalisation, where a difference between the formal definitions given and the approach used in practice cannot be so blithely elided as it can for human mathematicians.
Aside from issues of equality on objects, many common proof-techniques for reasoning about categories of multi-component structures can be expressed formally in terms of displayed categories, giving an essential toolbox for constructing and investigating such categories in formalisations.

To that end, most constructions and results of the present paper have been formalised in the proof assistant Coq, over the \UniMath library, with the goal of providing a practical library for re-use in further developments. 

While that development is in univalent type theory, for the present article we work in an ‘agnostic’ logical setting: all results may be understood either in type theory with univalence, or in a classical set-theoretic foundation.

\subsection{Outline}

We begin, in \textsection \ref{sec:background}, by laying out precisely the agnostic type-theoretic foundation in which we work, and recalling the basic background of category theory in this setting.

In \textsection \ref{sec:core}, we then set up the core definitions and constructions of displayed categories, along with various examples which will be used as running illustrations through the following sections.

Following this, in \textsection \ref{sec:limits}, we consider \emph{creation of limits}, a first simple example of a classical property of functors which can be stated and developed more cleanly as a property of displayed categories.

In \textsection \ref{sec:fibrations}, we move to the central such example: fibrations, along with their cousins isofibrations, discrete fibrations, and so on.
We set out the displayed-category definitions of these, and set out some of the basic results and constructions over this definition.

This provides a basis for the theory and application of fibrations in the type-theoretic setting.
In \textsection \ref{sec:comp-cats}, we use this to define \emph{comprehension categories}---a categorical axiomatisation of type dependency---bringing together several of the tools set up in earlier sections.

Finally, in \textsection \ref{sec:univalent}, we consider \emph{univalence} of displayed categories.
The main result there is that the total category of a univalent displayed category (suitably defined) over a univalent base category is univalent.
This generalises the \emph{structure identity principle} of \cite[\textsection 9.8]{HoTTbook}.

Throughout the article, many proofs would be almost word-for-word the same as standard proofs of the corresponding results about classically-defined fibrations (resp.\ creation of limits, etc), since displayed categories are exactly a formal abstraction of the language already used in such proofs.
We therefore omit these, to avoid repeating well-known material---but we invite the reader to recall the standard proofs, and see how directly they transfer.

Most other proofs are also either omitted or just briefly sketched, if they are either routine, available in detail in the formalisation, or both.

We follow Voevodsky in writing ‘Problem’, rather than ‘Theorem’, ‘Proposition’, etc., to denote proof-relevant results.
 
\subsection{Formalisation}

Most results of the present article have been formalised in Coq, over the \UniMath library of Voevodsky et al.~\cite{unimath}.

The primary goal of the formalisation is to provide a library for use in further work.
We have therefore focused in it on the results and constructions we expect to be useful in such work.
In particular, we have not formalised the comparisons with classical definitions: these are not needed for the development of fibrations etc.\ based on displayed categories, but rather form a justification that this approach is ‘correct’ from a classical point of view.

The formalisation is available as part of the \UniMath library, at \url{https://github.com/UniMath/UniMath}, 
in the subdirectory \href{https://github.com/UniMath/UniMath/tree/master/UniMath/CategoryTheory/DisplayedCats}{\nolinkurl{UniMath/CategoryTheory/DisplayedCats}}.
Instructions for use can be found in the repository’s \href{https://github.com/UniMath/UniMath/blob/master/README.md}{\nolinkurl{README.md}} file.

As a base for further development, readers are recommended to use the most up-to-date version of \UniMath.
However, organisation and naming of material there may change in future, so for permanent reference, the specific version described in this article is commit \href{https://github.com/UniMath/UniMath/tree/4dd5c1779d3b8db4db6392cfb4471985a1844d22}{4dd5c17} (8 December 2018), with browsable online documentation at \url{https://unimath.github.io/doc/UniMath/4dd5c17/toc.html}.

Definitions, constructions, and results included in the formalisation are labelled below with their corresponding identifiers, as e.g.\ \coqident{Core}{disp_cat},
and linked to their code in the reference version.

The material of the present paper constitutes about 5,000 lines of code.

\subsection{Revision notes}

This article is an expanded version of the conference paper \cite{ahrens_et_al:LIPIcs:2017:7722}, presented at Formal Structures for Computation and Deduction (FSCD) 2017.
Changes include the addition of Section~\ref{subsec:amnestic} on amnestic functors, and various minor local revisions.

\section{Background} \label{sec:background}

\subsection{Logical setting}

All the material of the present paper may be understood either in the univalent setting, or in classical set-theoretic foundations.

Precisely, our background setting throughout is Martin-Löf’s intensional type theory, with:
\(\Sigma\)-types, with the strong \(\eta\) rule;
identity types;
\(\Pi\)-types, also with \(\eta\), and functional extensionality; 
\(\constfont{0}\), \(\constfont{1}\), \(\constfont{2}\), and \(\constfont{N}\);
propositional truncation; and
two universes closed under all these constructions.

This setting is \emph{agnostic} about equality on types: it assumes neither univalence, nor UIP.
It is therefore expected to be compatible both with the addition of univalence, and with the interpretation of types as classical sets.

Some type-theoretic issues trivialise under the classical reading---for instance, the consideration of transport along equalities, which is unnecessary classically. 
Some topics also become less interesting there, as they admit only degenerate examples: in particular, the material on univalent categories.
The reader interested only in the classical setting may therefore ignore these aspects.

\subsection{Type-theoretic background}

We mostly follow the terminology standardised in the HoTT book \cite{HoTTbook}. 
A brief, but sufficient, overview is given in \cite{rezk_completion}, among other places.

We depart from it (and type-theoretic tradition in general) in writing just \emph{existence} for what is called \emph{mere existence} in \cite{HoTTbook}, since this is what corresponds (under the interpretation of types as sets) to the standard mathematical usage of \emph{existence}.

We will make frequent use of \emph{dependent paths/equalities} \cite[\textsection 6.2]{HoTTbook}
Specifically, in a type family \(B_x\) indexed by \(x:A\), we will write dependent equalities as e.g.\ \(p : y_0 \depeq[e] y_1\), where \(e : x_0 =_A x_1\) and \(y_i : B_{x_i}\).
We omit explicit mention of the type family \(B\), since it will always be clear from context.
The base \(A\) will often moreover be a set, in which case \(y_0 \depeq[e] y_1\) does not depend on the base path \(e\), so we suppress this and write just \(y_0 \depeq y_1\).

We will mostly ignore size issues; we would really like to think of everything as being universe-polymorphic.
For concreteness, however, \(\type\) may be understood always as the smaller of our two assumed universes, with types in this universe referred to as \emph{small}, and similarly \(\set\) as meaning the type or category of small sets, and so on.

\subsection{Categories}

We mostly follow the approach to category theory in the type-theoretic setting established in \cite{rezk_completion}.
We depart however from their terminology, writing \emph{categories} for what \cite{rezk_completion} calls precategories (since it is this that becomes the standard definition under the set interpretation), and writing \emph{univalent categories} for what \cite{rezk_completion} calls categories.

Specifically, in a \emph{category} \(\C\), the hom-sets \(\C(a,b)\) are required to be sets, but the type \(\C_0\) of objects is allowed to be an arbitrary type.
A category \(\C\) is \emph{univalent} if for all \(a, b : \C\), the canonical map \(\idtoiso_{a,b} : (a = b) \to \Iso[\C]{a}{b}\) is an equivalence: informally, if ‘equality of objects is isomorphism in \(\C\)’.

Following the \UniMath library, we write composition in the ‘diagrammatic’ order; that is, the composite of \(f : a \to b\) and \(g : b \to c\) is denoted \(\compose{f}{g} : a \to c\).

\section{Displayed categories} \label{sec:core}

In this section, we set out the basic definitions of displayed categories, displayed functors, and displayed natural transformations, along with key constructions on them, and examples which will act as running illustrations throughout the paper.

\subsection{Definition and examples}

\begin{definition}[\coqident{Core}{disp_cat}]
    Given a category \(\C\), a \emph{displayed category \(\D\) over \(\C\)} consists of
   \begin{enumerate}
    \item for each object \(c : \C\), a type \(\dob{\D}{c}\) of ‘objects over \(c\)’;
    \item for each morphism \(f : a \to b\) of \(\C\), \(x : \dob{\D}{a}\) and \(y : \dob{\D}{b}\), a set of ‘morphisms from \(x\) to \(y\) over \(f\)’, 
                  denoted \(\dmor{x}{f}{y}\) or \(\dto{x}{f}{y}\);
    \item for each \(c : \C\) and \(x : \dob{\D}{c}\), a morphism \(1_x : \dto{x}{1_c}{x}\);
    \item for all morphisms \(f : a \to b\) and \(g : b \to c\) in \(\C\) and objects \(x : \dob{\D}{a}\) and \(y : \dob{\D}{b}\) and \(z : \dob{\D}{c}\),
          a function 
          \[\dmor{x}{f}{y} \times \dmor{y}{g}{z} \to \dmor{x}{\compose{f}{g}}{z} \, , \]
          denoted like ordinary composition by \((\f,\g) \mapsto \compose{\f}{\g} : \dto{x}{\compose{f}{g}}{z}\), where \(\f : \dto{x}{f}{y}\) and \(\g : \dto{y}{g}{z}\),
   \end{enumerate} \saveitem
   such that, for all suitable inputs, we have:
   \begin{enumerate} \restoreitem
	  \item \(\compose{\f}{1_y} \depeq \f,\)
	  \item \(\compose{1_x}{\f} \depeq \f,\)
	  \item \(\compose{\f}{(\compose{\g}{\h})} \depeq \compose{(\compose{\f}{\g})}{\h}.\)
   \end{enumerate}
\end{definition}

Note that the axioms are all \emph{dependent} equalities, over equalities of morphisms in \(\C\): for instance, if \(\f : \dto{x}{f}{y}\), then \(\compose{\f}{1_y} : \dto{x}{\compose{f}{1_b}}{y}\), so the displayed right unit axiom \(\compose{\f}{1_y} \depeq \f\) is over the ordinary right unit axiom \(\compose{f}{1_b} = f\) of \(\C\).
This will be typical in what follows: equations in displayed categories will be modulo analogous equations in \(\C\), which we will usually suppress without further comment.

As promised, any displayed category over \(\C\) induces an ordinary category over \(\C\):

\begin{definition}[\coqident{Core}{total_category}, \coqident{Core}{pr1_category}]
 Let \(\D\) be a displayed category \(\D\) over \(\C\).
 The \emph{total category} of \(\D\), written \(\total{\D}\) (or \(\total[\C]{\D}\), or \(\total[c:\C]{\dob{\D}{c}}\))
 is defined as follows:
   \begin{enumerate}
    \item objects are pairs \((a,x)\) where \(a : \C\) and \(x : \dob{\D}{a}\); in other words,
          the type of objects is
          \[ (\total{\D})_0 \defeq \sum_{ a : \C} \dob{\D}{a}\, ; \]
    \item morphisms \((a,x) \to (b,y)\) are pairs \((f,\f)\) where \(f : a \to b\) and 
          \(\f : \dto{x}{f}{y}\); in other words,
          \[ (\total{\D})\bigl( (a,x)  , (b,y) \bigr) \defeq \sum_{ f : \C(a,b) } \dmor{x}{f}{y}\, ; \]
   \item composition and identities in \(\total{\D}\) are induced straightforwardly from those of \(\C\) and \(\D\), and similarly for the axioms.
   \end{enumerate}
   The evident \emph{forgetful functor} \(\forget[\D] : \total{\D} \to \C\) simply takes the first projection, on both objects and morphisms.
\end{definition}

\begin{example}[\coqident{Examples}{group.disp_grp}]
  The \emph{category of groups} can be defined as the total category of a displayed category \(\dGrp\), over \(\set\):
  \begin{enumerate}
  \item \(\dob{\dGrp}{X}\) is the set of group structures on the set \(X\);
  \item given a function \(f : X \to Y\) and group structures \((\mu,e)\) on \(X\) and \((\mu',e')\) on \(Y\), 
    \(\dmor{(\mu,e)}{f}{(\mu',e')}\) is (the type representing) the proposition ‘\(f\) is a homomorphism with respect to \((\mu,e)\), \((\mu',e')\)’;
  \item the displayed composition ‘operation’ is the fact that the composite of homomorphisms is a homomorphism; similarly for the identity;
  \item the axioms are trivial, since the displayed hom-sets are propositions.
  \end{enumerate}
  The total category of this is exactly the usual category of groups.
\end{example}

\begin{example}[\href{\coqdocbasebaseurl/UniMath.Topology.CategoryTop.html\#disp_top}{\nolinkurl{disp_top}}]
  The category of topological spaces can be defined as the total category of the displayed category \(\dTop\) over \(\set\):
  \begin{enumerate}
   \item \(\dob{\dTop}{X}\) is the set of topologies on the set \(X\);
   \item given a function \( f : X \to Y\) and topologies \(T\) on \(X\) and \(T'\) on \(Y\),
          \(\dmor{T}{f}{T'}\) is the proposition ‘\(f\) is continuous with respect to \(X\) and \(Y\)’.
  \end{enumerate}
\end{example}

\begin{example}[\coqident{Examples}{disp_over_unit}]
  Any category can be viewed as a displayed category over the terminal category.
\end{example}

\begin{example}[\coqident{Constructions}{disp_full_sub}] \label{ex:full-subcat-displayed}
  Let \(P : \C \to \type\) be a (type-valued) predicate on the objects of \(\C\).
  Then there is an associated displayed category, with object family exactly \(P\), and with \(\dmor{y}{f}{y'} \defeq 1\) for all \(f : c \to c'\), \(y:P c\), and \(y' : P c'\).
  The operations and axioms are trivial.
  
  Its total category is the full subcategory of \(\C\) of objects satisfying the predicate \(P\).
\end{example}

Properties of the forgetful functor can often be straightforwardly read off from the displayed category:
\begin{proposition}[\coqident{Core}{full_pr1_category}, \coqident{Core}{faithful_pr1_category}]
  Let \(\D\) be a displayed category over \(\C\).
  If every displayed hom-set \(\dmor{y}{f}{y'}\) of \(\D\) is a proposition (resp.\ inhabited, contractible) then \(\forget : \total{\D} \to \C\) is faithful (full, fully faithful). \qed
\end{proposition}

Besides the total category, a displayed category also possesses \emph{fibre} categories:
\begin{definition}[\coqident{Constructions}{fiber_category}]
  Given a displayed category \(\D\) over \(\C\), and an object \(c : \C\), define the \emph{fibre category} \(\fiber{\D}{c}\) of \(\D\) over \(c\) as the category
  with objects \(\dob{\D}{c}\) and with morphisms \(\mor{x}{y} \defeq \dmor{x}{1_c}{y}\). 
  Composition and identity are induced by that of \(\D\).
\end{definition}
In general these may not be so well behaved as the total category; they will typically be interesting and well-behaved just when \(\D\) is an isofibration (Definition~\ref{def:isofib}).

\begin{remark}
  In choosing notation and terminology for examples of displayed categories, a question arises: should one name displayed categories according to their total category, or according to their fibres?

  This problem arises already with fibrations in the classical setting; so we follow for the most part the usual compromises used there.
  Specifically, when a given total category has a particularly canonical displaying—for example, groups displayed over sets—we will use the same name for the displayed category and its total category, so for example \(G : \Grp\) denotes a group, while \((\mu,e) : \dob{\dGrp}{X}\) is a group structure on \(X\).
  On the other hand, when different displayed categories have equivalent total categories—for instance, the product \(\C \times \C'\) may be displayed over either \(\C\) or \(\C'\)—then we will adopt different notation to distinguish these, usually based on the resulting fibre categories.
\end{remark}

Other examples we will meet below include:
\begin{enumerate}
  \item any product \(\C \times \C'\), displayed over its first factor as \(\dconst[\C]{\C'}\) (Example~\ref{ex:const-displayed});
  \item the arrow category \(\arr{\C}\), in several ways: displayed over \(\C^2\), with fibres hom-sets; and displayed over \(\C\), with fibres either the slices or the coslices of \(\C\) (Example~\ref{ex:arr-displayed});
  \item categories of algebras for endofunctors and monads (Examples~\ref{ex:endofunctor-alg-displayed}, \ref{ex:monad-alg-displayed}).
\end{enumerate}
We postpone their full definitions until we have a few more tools set up.

\begin{remark}
  Equivalent definitions in a similar vein as displayed categories—that is, ‘fibred’ presentations of arbitrary functors into a fixed base category—can be recovered from more sophisticated categorical structures in several ways:
  as lax 2-functors or double functors from the base category into the bicategory or double category of spans,
  or as normal lax 2-functors/double functors into the bicategory/double category of distributors (as observed by Bénabou in \cite[\textsection 7]{benabou:distributors-at-work}),
  or as double profunctors from the base category to the terminal double category.%
  \footnote{Our thanks to Mike Shulman and an anonymous referee for pointing out some of these reformulations.}
\end{remark}

\subsection{Displayed functors and natural transformations}

Another occurrence of equality of objects is in various definitions where diagrams of functors are assumed to commute on the nose.
For instance, comprehension categories involve a fibration \(p : \T \to \C\), and a functor \(\chi : \T \to \arr{\C}\), such that \(\compose{\chi}{\cod} = p\) \cite[Theorem 9.3.4]{jacobs-cat-logic}; 
similar conditions occur in the definition of functorial factorisations, in the theory of weak factorisation systems (among many other places).

It is typically clear that the definitions could also be phrased without equality of objects, at some cost in concision or clarity.
Indeed, they are almost always of the form \(\compose{G}{\forget[\D']} = F\), where \(G\) is a functor into the total category of some displayed category, and \(F\) is a previously-given functor into the base.
They are often furthermore of the more specialised form \(\compose{G}{\forget[\D']} = \compose{\forget[\D]}{F}\).

By axiomatising this situation, as \emph{displayed functors} over functors into the base, such definitions can be stated without equality of objects, with no loss of clarity.

\begin{definition}[\coqident{Core}{disp_functor}] \label{def:functor}
  Let \(F : \C \to \C'\) be a functor, and \(\D\), \(\D'\) displayed categories over \(\C\) and \(\C'\) respectively.
  A \emph{(displayed) functor \(G\) from \(\D\) to \(\D'\) over \(F\)} consists of:
  \begin{enumerate}
   \item maps \(\dob{G}{c} : \dob{\D}{c} \to \dob{\D'}{Fc}\), for each \(c:\C\) (which we usually write just as \(G\), omitting \(c\)); and 
   \item maps \(\dmor{x}{f}{y} \to \dmor{Gx}{Ff}{Gy}\), for each \(f : c \to c'\) in \(\C\);
   \item satisfying the evident dependent analogues of the usual functor laws.
  \end{enumerate}
\end{definition}

  A displayed functor \(G\) over \(F\) straightforwardly induces a \emph{total functor} between total categories, written \(\total{G} : \total{\D} \to \total{\D'}\), such that \(\compose{\total{G}}{\forget[\D']} = \compose{\forget[\D]}{F}\).
  Indeed, displayed functors are precisely equivalent to such functors between total categories.
  We often therefore call the total functor just \(G\).
  
  Similarly, a functor \(G\) over \(F\) induces \emph{fibre functors} \(\fiber{G}{c} : \fiber{\D}{c} \to \fiber{\D'}{Fc}\), for each \(c : \C\).

  A useful special case is when \(F\) is the identity functor of \(\C\), in which case we call \(G\) just a functor over \(\C\); this is precisely equivalent to a functor between the total categories strictly over \(C\) in the usual sense.

\begin{definition}[\coqident{Core}{disp_nat_trans}]\label{def:nat_trans}
 Let \(F, F' : \C \to \C'\) be functors, \(\alpha : F \to F'\) a natural transformation,
 and \(G\) and \(G'\) displayed functors from \(\D\) to \(\D'\) over \(F\) and \(F'\) respectively.
 A \emph{displayed natural transformation \(\beta\) from \(G\) to \(G'\) over \(\alpha\)} consists of
 \begin{enumerate}
  \item for each \(c : \C\) and \(d : \dob{\D}{c}\), a morphism \(\beta(d) : \dmor{G(d)}{\alpha(c)}{G'(d)}\)
  \item such that for any \(f : \C(c,c')\) and \(\f : \dmor{d}{f}{d'}\), \(\compose{G\f}{\beta(d')} \depeq \compose{\beta(d)}{G'\f}.\)
 \end{enumerate}
\end{definition}

Just as ordinary functors and natural transormations form a functor category, their displayed versions form a displayed category over the functor category between the bases:

\begin{definition}[\coqident{Constructions}{disp_functor_cat}]
 Given categories \(\C\) and \(\C'\), and displayed categories \(\D\) and \(\D'\) over \(\C\) and \(\C'\) respectively,
 there is a displayed category \([\D,\D']\) over \([\C,\C']\), defined as follows:
 \begin{enumerate}
  \item objects over \(F : \C \to \C'\) are displayed functors from \(\D\) to \(\D'\) over \(F\);
  \item morphisms over \(\alpha : F \to F'\) from \(G\) to \(G'\) are displayed natural transformations from \(G\) to \(G'\) over \(\alpha\);
  \item composition and identity are given by pointwise composition and identity.
 \end{enumerate}
\end{definition}

Displayed analogues of usual lemmas on the functor category hold; for instance:
\begin{lemma}[\coqident{Constructions}{is_disp_functor_cat_iso_iff_pointwise_iso}]
  A displayed natural transformation is an isomorphism in the displayed functor category if and only if it is
  an isomorphism pointwise. \qed
\end{lemma}

We could now go on and define \emph{displayed adjunctions} over adjunctions between the bases,
\emph{displayed equivalences} over equivalences of the base, and so on.
From these, one gets adjunctions and equivalences, respectively, of total categories.
A very useful special case is that of displayed adjunctions and equivalences over the identity in the base,
yielding adjunctions and equivalences of total categories leaving the first components of objects untouched.

These definitions are provided in the formalisation; indeed, the original motivation of the present work and formalisation was to have these available, in order to construct an equivalence of univalent categories between CwF-structures and split type-category structures on a fixed base category (cf.\ the equivalence of types of \cite[Construction~3.19]{lmcs:4814}).
However, an account of this is beyond the scope of the present paper.

One may also naturally ask what structure the total collections of displayed categories, functors, and natural transformations form.
We expect that they should form a bicategory when the base category is held fixed, and more generally a displayed bicategory over the bicategory of categories; but this again is beyond the scope of the present work.
 
\subsection{Constructions on displayed categories}

To efficiently construct our remaining key examples, we set up some basic general constructions on displayed categories.

\begin{definition}[\coqident{Core}{reindex_disp_cat}]
  Let \(\D\) be a displayed cat over \(\C\), and \(F : \C' \to \C\) a functor.
  Then \(\dreind{F}{\D}\), the \emph{pullback of \(\D\) along \(F\)}, is the displayed category over \(\C'\) defined by
  \begin{enumerate}
  \item \(\dob{(\dreind{F}{\D})}{c} \defeq \dob{\D}{Fc}\)
  \item \(\dmor[\dreind{F}{\D}]{d}{f}{d'} \defeq \dmor[\D]{d}{Ff}{d'}\)
  \end{enumerate}
  with the evident composition and identities.
  There is an evident displayed functor \(\dreind{F}{\D} \to \D\) over \(F\).
\end{definition}

\begin{example}[\coqident{Examples}{disp_cartesian}] \label{ex:const-displayed}
  Given any categories \(\C\), \(\C'\), the \emph{constant displayed category over \(\C\) with fibre \(\C'\)}, denoted \(\dconst[\C]{\C'}\) (or just \(\dconst{\C'}\), when \(\C\) is implicit), is the pullback along the unique functor \(\C \to 1\) of \(\C'\), seen as a displayed category over \(1\).

  There is an evident equivalence from the total category \(\total[\C]{\dconst{\C'}}\) to the product \(\C \times \C'\), strictly over \(\C\).
\end{example}

\begin{definition}[\coqident{Constructions}{sigma_disp_cat}]
  Let \(\D\) be a displayed category over \(\C\), and \(\E\) a displayed category over \(\total[\C]{\D}\).
  The \emph{\(\Sigma\)-category of \(\E\) over \(\D\)}, denoted \(\Sigdisp{\D}{\E}\), is the displayed category over \(\C\) defined as follows:
  \begin{enumerate}
  \item \(\dob{\left(\Sigdisp{\D}{\E}\right)}{x} \defeq \sum_{y:\dob{\D}{x}} \dob{\E}{(x,y)}\)
  \item \(\dmor{(y,z)}{f}{(y',z')} \defeq \sum_{\f:\dto{y}{f}{y'}} \dmor{z}{(f,\f)}{z'}\)
  \item operations defined componentwise from those of \(\D\) and \(\E\).
  \end{enumerate}
  There is an evident equivalence of total categories \(\total[{\left(\total[\C]{\D}\right)}]{\E} \to \total[\C]{\left(\Sigdisp{\D}{\E}\right)}\) over \(\C\).
\end{definition}

\begin{example}[\coqident{Examples}{disp_arrow}, \coqident{Examples}{disp_domain}, \coqident{Codomain}{disp_codomain}] \label{ex:arr-displayed}
  The arrow category has three different displayed incarnations:
  \begin{enumerate}
  \item By \(\darr{\C}\), we mean the displayed category over \(\C \times \C\) with
    \begin{enumerate}
    \item \(\dob{\darr{\C}}{x,y} \defeq \mor[\C]{x}{y}\)
    \item \(\dmor[\darr{\C}]{f}{h,k}{g} \defeq (\compose{f}{k} = \compose{h}{g})\), i.e.\ the proposition that the resulting square commutes.
    \end{enumerate}
    As our notation suggests, the total category of this is the usual arrow category of \(\C\).

  \item Pulling this back along the canonical equivalence \(\total[\C]{\dconst{\C}} \to \C \times \C\), and taking the \(\Sigma\)-category of the result, we obtain a displayed category over \(\C\) which we denote \(\coslice{\C}\), since its fibre categories are just the \emph{co-slices} of \(\C\).
    Its total category is equivalent over \(\C\) to \(\dom : \arr{\C} \to \C\).

  \item If in the previous example, we instead pull back along the equivalence \(\total[\C]{\dconst{\C}} \to \C \times \C\) that swaps the two components, we get instead the displayed category of \emph{slices} of \(\C\), with total category equivalent over \(\C\) to \(\cod : \arr{\C} \to \C\).
        \label{item:slice}
  \end{enumerate}
\end{example}

\begin{example}[\coqident{Examples}{disp_cat_functor_alg}] \label{ex:endofunctor-alg-displayed}
  Suppose \(F : \C \to \C\) is an endofunctor.
  Then \emph{\(F\)-algebras} naturally form a displayed category \(\dAlg[\C]{F}\) over \(\C\), with
  \begin{enumerate}
  \item \(\dob{\dAlg{F}}{c} \defeq \mor[\C]{Fc}{c}\)
  \item \(\dmor{\alpha}{f}{\beta} \defeq (\compose{\alpha}{f} = \compose{Ff}{\beta})\), i.e.\ the proposition that \(f : a \to b\) is an algebra homomorphism \((a,\alpha) \to (b,\beta)\).
  \end{enumerate}
  The total category is the usual category \(\Alg[\C]{F}\).
  We will sometimes write \(\EndAlg[\C]{F}\) to distinguish this from categories of \emph{monad} algebras.
\end{example}

\begin{example}[\coqident{Examples}{disp_cat_monad_alg}] \label{ex:monad-alg-displayed}
  Suppose \((T,\mu,\eta)\) is a monad on \(\C\).
  The full subcategory of \(\EndAlg[\C]{T}\) consisting of the \emph{monad} algebras for \((T,\mu,\eta)\) can be seen as a displayed category over \(\EndAlg[\C]{T}\), as in Example~\ref{ex:full-subcat-displayed}.
  Taking the \(\Sigma\)-category of this yields the monad-algebras \(\MonadAlg[\C]{(T,\mu,\eta)}\) as a displayed category over \(\C\).
  
  As usual, we write just \(\Alg[\C]{T}\) when there is no risk of confusion.
\end{example}

\section{Creation of limits} \label{sec:limits}

\emph{Creation of limits} is our first example of a concept which can be profitably reformulated in terms of displayed categories.

As a property of functors, it is a standard and fruitful tool in category theory.
It has however often been viewed with some mistrust for involving equalities of objects: see, for example, \cite[Remark 5.3.7]{leinster}.

If formulated instead as a property of displayed categories, it involves no equalities of objects:

\begin{definition}[\coqident{Limits}{creates_limit}]
  Let \(\D\) be a displayed category over \(\C\), \(J\) a graph, and \(F\) a diagram of shape \(J\) in \(\total{\D}\).
  Given a limiting cone \(\lambda\) for the diagram \(\compose{F}{\forget} : J \to \C\) in \(\C\), with vertex \(c : \C\), we say that \(\D\) \emph{creates a limit for \(F\) over \(\lambda\)} if
  \begin{enumerate}
  \item there is a unique cone on \(F\) over \(\lambda\); that is, a unique object \(d : \dob{\D}{c}\) and family of arrows \(\mu_j : \dmor{d}{\lambda_j}{\pi_2 F(j)}\) such that the pairs \((\lambda_j,\mu_j)\) form a cone on \(F\) in \(\total{\D}\);
  \item and, furthermore, this unique cone \((\lambda_j,\mu_j)_{j:J}\) is limiting.
  \end{enumerate}
  More generally, we say that \(\D\) \emph{creates limits of shape \(\J\)} (or \emph{creates small limits}, etc.)\ if, for any diagram \(F\) as above over \(\J\) (resp.\ over any small \(\J\)), and every limiting cone \(\lambda\) on \(\compose{F}{\forget}\) in \(\C\), \(\D\) creates a limit for \(F\) over \(\lambda\).
\end{definition}

It is routine to check that this does indeed correspond to the standard notion:
\begin{proposition}
  A displayed category \(\D\) over a category \(\C\) creates a limit or class of limits, in our sense, if and only if the functor \(\forget[\D] : \total{\D} \to \C\) does so in the classical sense. \qed
\end{proposition}

It of course follows immediately from this that the displayed definition implies the various standard consequences of creation of limits.
In fact, however, the proofs from the displayed definition are at least as direct as the standard proofs; for instance,
\begin{proposition}[\coqident{Limits}{total_limits}, \coqident{Limits}{pr1_preserves_limit}]
  Suppose the category \(\C\) has limits of shape \(\J\), and the displayed category \(\D\) over \(\C\) creates limits of shape \(\J\).
  Then \(\total{\D}\) has all such limits, and \(\forget[\D] : \total{\D} \to \C\) preserves them. \qed
\end{proposition}

Moreover, all the main standard examples of functors that create limits can be seen as the forgetful functors associated to displayed categories.

\begin{example}[\coqident{Examples}{creates_limits_functor_alg}]
  For any endofunctor \(F : \C \to \C\), the displayed category of \(F\)-algebras over \(\C\) creates all limits.
  Likewise, for any monad \(T\) on \(\C\), the displayed category of \(T\)-algebras over \(\C\) creates all limits.
\end{example}

\section{Fibrations} \label{sec:fibrations}

We consider, in this section, three important variations of fibrations of categories: Grothen\-dieck fibrations (and their dual, opfibrations); isofibrations; and discrete fibrations.

We depart from some classical literature in defining fibrations by default to be \emph{cloven}---that is, to include an operation providing all lifts required.
(This is not novel: it has been preferred also by other authors, to avoid indiscriminate use of the axiom of choice.) 
We distinguish the case where liftings are merely known to exist as \emph{weak} fibrations.

\subsection{Fibrations and opfibrations}

\begin{definition}[\coqident{Fibrations}{is_cartesian}]
  Let \(\D\) be a displayed category over \(\C\).
  A map \(\f : \dmor{d'}{f}{d}\) of \(\D\) over \(f : c' \to c\) is \emph{cartesian} if
  for each \(g : c'' \to c'\), \(d'' : \dob{\D}{c''}\), and \(\h : \dmor{d''}{\compose{g}{f}}{d}\),
  there is a unique \(\g : \dmor{d''}{g}{d'}\) such that \(\compose{\g}{\f} = \h\).
\end{definition}

\begin{definition}[\coqident{Fibrations}{cartesian_lift}]
  Let \(\D\) be a displayed category over \(\C\).
  A \emph{cartesian lift} of \(f : \C(c',c)\) and \(d : \dob{\D}{c}\) consists of an object \(d' : \dob{\D}{c'}\) and a cartesian map \(\f : \dmor{d'}{f}{d}\).
\end{definition}

\begin{definition}[\coqident{Fibrations}{cleaving}, \coqident{Fibrations}{fibration}, \coqident{Fibrations}{weak_fibration}]
  A \emph{cleaving} for a displayed category \(\D\) over a category \(\C\) is a function giving, for each \(f : c' \to c\) and \(d : \dob{\D}{c}\), a cartesian lift of \(f\) and \(d\).
  A \emph{(cloven) fibration} over \(\C\) is a displayed category equipped with a cleaving.
  A \emph{weak fibration} is a displayed category such that for each such \(f\), \(d\) as above, there exists some cartesian lift.
\end{definition}

All the above have evident duals: opcartesian maps and lifts, and weak/cloven opfibrations.
Again, these all correspond straightforwardly to their classical versions:
\begin{proposition}
  A map in a total category \(\total[\C]{\D}\) is cartesian in our sense (resp.\ opcartesian) exactly if it is cartesian (opcartesian) with respect to \(\forget[\D]\) in the classical sense.
  A displayed category \(\D\) is a cloven (resp.\ weak) fibration in our sense exactly if \(\forget[\D]\) is one in the classical sense (i.e.\ \cite[Def.~5.3.5]{leinster}, read unchanged in the univalent setting). \qed
\end{proposition}

As with the standard definition, cartesian lifts are unique up to isomorphism.
Proposition~\ref{prop:unique-lifts} below shows that when \(\D\) is univalent, they are literally unique.

An important example in our applications of interest is the arrow category:

\begin{proposition}[\coqident{Codomain}{cartesian_iff_isPullback}] \label{prop:cartesian-iff-pullback}
  For any category \(\C\), consider the displayed category \(\slice{\C}\) of slices of \(\C\), as in Example~\ref{ex:arr-displayed}.\ref{item:slice} above.
  An arrow \(h : \dto{f}{k}{g}\) in \(\slice{\C}\) is cartesian exactly if its associated commuting square is a pullback.
  The displayed category \(\slice{\C}\) is a weak fibration just if all pullbacks exist in \(\C\), and a (cloven) fibration just if \(\C\) has chosen pullbacks. \qed
\end{proposition}

Finally, we transfer the definition of \emph{split} fibrations.
It seems likely to us that---as with the hom-set condition for categories---split fibrations in the type-theoretic setting should include a setness condition in order to be as useful and well-behaved as classically:
\begin{definition}[\coqident{Fibrations}{is_split}]
  Say a fibration \(\D\) over \(\C\) is \emph{split} if:
  \begin{enumerate}
  \item each \(\dob{\D}{c}\) is a set; and
  \item the chosen lifts of identities are identities, and the chosen lift of any composite is the composite of the individual lifts.
  \end{enumerate}
\end{definition}

\subsection{Isofibrations}

\begin{definition}[\coqident{Core}{iso_disp}]
  Let \(\D\) be a displayed category over \(\C\), and \(f : c \iso c'\) an isomorphism in \(\C\).
  
  A map \(\f : \dmor{d}{f}{d'}\) is a \emph{(displayed) isomorphism} if it has a 2-sided inverse,
  i.e.\ some \(\g : \dmor{d'}{\inv{f}}{d}\)
  such that \(\compose{\f}{\g} \depeq 1_d\) and \(\compose{\g}{\f} \depeq 1_{d'}\).
  We write \( \f : \diso{d}{f}{d'}\).
\end{definition}

As with ordinary isomorphisms, the inverse of a displayed isomorphism is unique.

\begin{definition}[\coqident{Fibrations}{weak_iso_fibration}, \coqident{Fibrations}{iso_cleaving}, \coqident{Fibrations}{iso_fibration}]\label{def:isofib}
  Let \(\D\) be a displayed category over \(\C\).
  Say \(\D\) is a \emph{weak isofibration} if for each isomorphism \(i : c' \iso c\) in \(\C\) and \(d : \dob{\D}{c}\),
  there exists some object \(d' : \dob{\D}{c'}\) and isomorphism \(\bar{i} : \diso{d'}{i}{d}\).
  An \emph{iso-cleaving} on \(\D\) is a function giving, for each such \(i, d\), some such \(d', \bar{i}\).
  A \emph{(cloven) isofibration} over \(\C\) is a displayed category equipped with an iso-cleaving.
\end{definition}

\begin{proposition}A displayed category is a weak (resp.\ cloven) isofibration in our sense just if its forgetful functor is one in the classical sense. \qed
\end{proposition}

\begin{example}[\coqident{Examples}{iso_cleaving_functor_alg}]
  The displayed categories of groups, topological spaces, and similar are all naturally isofibrations over \(\set\), just as classically.
  More generally, so are the displayed categories of algebras for endofunctors and monads.
\end{example}

In fact, in the univalent setting, isofibrations often come for free:

\begin{problem}[\coqident{Fibrations}{iso_cleaving_category}]\label{problem:univalent_isofib}
  Let \(\D\) be a displayed category over a univalent category \(\C\).
  Then \(\D\) is an isofibration.
\end{problem}

\begin{construction}{problem:univalent_isofib}
  Since \(\C\) is univalent, every isomorphism \(i : c' \iso c\) is uniquely of the form \(\idtoiso(e)\).
  To give an iso-cleaving on \(\D\), it therefore suffices to give, for each \(e : c' = c\) and \(d : \dob{D}{c}\), some \(d':\dob{D}{c'}\) and lift \(\bar{i} : \diso{d'}{\idtoiso(e)}{d}\).
  By identity elimination, the case \(e \defeq \idpath[c]\) suffices; in this case, we take \(d' \defeq d\) and \(\bar{i} \defeq 1_d\).
\end{construction}

  Assuming the univalence axiom, the examples above of \(\dGrp\) and \(\dTop\) over \(\set\) therefore come for free.
  However, we note them separately (and prove them directly, in the formalisation), both to show that they do not require univalence, and to have their action explicitly.

\begin{remark}
  As the examples given illustrate, most fibrations and isofibrations encountered in nature are categories/functors that arise as the total category/forgetful functor of a displayed category.
  This, we argue, supports the idea that it is natural to take the displayed-category definitions as basic for developing fibrations and related notions, especially in the type-theoretic setting.

  However, not all examples are of this form.
  For instance, suppose \(F : \C' \to \C\) is a functor of small categories that is a complemented inclusion on objects; then the precomposition functor \(F^* : \psh{C} \to \psh{C'}\) between their presheaf categories is an isofibration.
  However, in the classical setting, \(\psh{C}\) is not literally the total category of any displayed category over \(\psh{C'}\) (though it is of course isomorphic to one).
\end{remark}

\subsection{Discrete fibrations}

\begin{definition}[\coqident{Fibrations}{is_discrete_fibration}]
  Let \(\D\) be a displayed category over \(\C\).
  Say that \(\D\) is a \emph{discrete fibration} if 
  \begin{enumerate}
  \item for each \(c : \C\), the type \(\dob{\D}{c}\) is a set; and
  \item for any \(f : \C(c',c)\) and \(d : \dob{\D}{c}\), there is a unique
    \(d' : \dob{\D}{c'}\) and \(\f : \dmor{d'}{f}{d}\).
  \end{enumerate}
\end{definition}

These lifts are automatically cartesian; so any discrete fibration is canonically a fibration (\coqident{Fibrations}{fibration_from_discrete_fibration}), and is moreover split (\coqident{Fibrations}{is_split_fibration_from_discrete_fibration}).

Thanks to the setness condition, discrete fibrations over a fixed base category \(\C\) and displayed functors between them form a category; and, just as classically, we have:
\begin{problem}[\coqident{Fibrations}{forms_equivalence_disc_fib}] \label{problem:equiv-disc-psh}
 For any category \(\C\), there is a (strong) equivalence of categories between \(\psh{\C}\) and the category of discrete fibrations over \(\C\). \qed
\end{problem}

For a presheaf \(P\) on \(\C\), the classical category of elements of \(P\) is the total category of the displayed discrete fibration given by the above equivalence.

\section{Comprehension categories} \label{sec:comp-cats}

We now turn briefly to \emph{comprehension categories} and \emph{categories with attributes}, just as a glimpse of the applications in semantics of type theory which provided the proximate motivation for the present development.

\begin{definition}[\coqident{ComprehensionC}{comprehension_cat_structure}] \label{def:comprehension_cat}
  A \emph{comprehension category} consists of a category \(\C\), a fibration \(\T\) over \(\C\), and a functor \(\chi : \T \to \slice{\C}\) over \(\C\) (the ‘comprehension’) preserving cartesian arrows.
\[
 \begin{xy}
   \xymatrix{
            \total{\T} \ar[dr]_{\forget} \ar[rr]^{\chi} & & \total[c:\C]{\slice[c]{\C}} \ar[dl]^{\forget} \\
               & \C  &
            }
 \end{xy}
\]
\end{definition}

This is almost identical to \cite[Definition 2.1.1]{lumsdaine-warren:local-universes}, modulo the correspondence between displayed categories/functors and ordinary categories/functors over the base.  As such, it is a direct reformulation of the original definition \cite[Definition 10.4.2]{jacobs-cat-logic}, taking the fibration of types as primary.

\begin{definition}
  A \emph{split type-category} (aka \emph{category with attributes}) consists of 
  a category \(\C\); 
  a presheaf \(\Ty\) on \(\C\); 
  an operation assigning to each \(\Gamma : \C\) and \(A : \Ty[\Gamma]\) an object and map \(\pi_A : \compext{\Gamma}{A} \to \Gamma\); 
  and operations giving, for each \(f : \Gamma' \to \Gamma\) and \(A : \Ty[\Gamma]\), a map \(\compext{f}{A} : \compext{\Gamma'}{f^*A} \to \compext{\Gamma}{A}\) exhibiting \(\pi_{f^*A}\) as a pullback of \(\pi_A\):
  \[
   \begin{xy}
    \xymatrix@C=4pc{
                     \compext{\Gamma'}{f^*A} \ar[d]_{\pi_{f^*A}} \ar[r]^{\compext{f}{A}} \pb &   \compext{\Gamma}{A} \ar[d]^{\pi_A}
                      \\
                      \Gamma' \ar[r]_{f} & \Gamma .
    }
   \end{xy}
  \]
\end{definition}

\begin{problem}[{\cite[Thm.~2.3]{blanco}}]\label{problem:compcat_from_cdm}
  Any category with attributes induces a comprehension category with the same base.
\end{problem}

\begin{construction}{problem:compcat_from_cdm}
  The equivalence of Problem~\ref{problem:equiv-disc-psh} turns \(\Ty\) into a (discrete) fibration.
  The operations \(\compext{\Gamma}{A}\), \(\pi_A\), and \(\compext{f}{A}\) provide the action on objects and arrows of the comprehension functor; while the pullback condition, combined with Proposition~\ref{prop:cartesian-iff-pullback}, ensures that it preserves cartesian maps.
\end{construction}

\section{Univalence and the Structure Identity Principle} \label{sec:univalent}

\subsection{Displayed univalence}

\begin{definition}[\coqident{Core}{idtoiso_disp}]
  Let \(\D\) be a displayed category over \(\C\). 
  Given \(c, c' : \C\), \(e : c = c'\), \(d : \dob{\D}{c}\), \(d' : \dob{\D}{c'}\),
  and \(e' : d \depeq[e] d'\), we write \(\idtoiso(e,e') : \diso{d}{\idtoiso(e)}{d'}\) 
  for the canonical displayed isomorphism obtained by identity elimination on \(e\), \(e'\).
\end{definition}
Note that we overload the notation \idtoiso, using it for both ordinary and displayed categories.

\begin{definition}[\coqident{Core}{is_univalent_disp}]
  Let \(\D\) be a displayed category over \(\C\). 
  Say that \(\D\) is \emph{univalent} if for any \(c, c' : \C\) and \(e : c = c'\) and \(d : \dob{\D}{c}\) and \(d' : \dob{\D}{c'}\),
  the above map \((d \depeq[e] d') \to \dIso{d}{\idtoiso(e)}{d'}\) is an equivalence.
\end{definition}

To verify univalence of a displayed category, it clearly suffices to prove the condition just in the case where \(e\) is reflexivity.
But displayed isomorphisms over identities are just isomorphisms in the fibre categories, so we have:

\begin{proposition}[\coqident{Constructions}{is_univalent_disp_iff_fibers_are_univalent}]
  Let \(\D\) be a displayed category over \(\C\).
  Then \(\D\) is univalent exactly if each of its fibre categories is univalent. \qed
\end{proposition}

The key practical application of displayed univalence is in proving that complex categories built up using displayed categories are univalent:

\begin{theorem}[\coqident{Core}{is_univalent_total_category}]\label{theorem:total_univalent}
 Let \(\C\) be a univalent category, and let \(\D\) be a univalent displayed category over \(\C\).
 Then the total category \(\total{\D}\) is univalent. \qed
\end{theorem}

However, displayed univalence is a meaningful notion even when the base is not known to be univalent; one has, for instance:

\begin{proposition}[\coqident{Fibrations}{isaprop_cartesian_lifts}, \coqident{Fibrations}{univalent_fibration_is_cloven}] \label{prop:unique-lifts}
  Let \(\D\) be a \emph{univalent} displayed category over \(\C\).
  For any \(f : c' \to c\) and \(d : \dob{\D}{c}\), if a cartesian lift \((d',\f)\) of \(f\) and \(d\) exists, then it is unique; that is, the type of cartesian lifts is a proposition.
  More generally, if \(\D\) is a weak (iso-)fibration, then it possesses a unique (iso-)cleaving.
\end{proposition}

\begin{proof}
  The usual classical argument shows that cartesian lifts are unique up to isomorphism.
  By univalence of \(\D\), it follows that they are literally unique.
  
  It follows that the type of (iso-)cleavings of \(\D\) is a proposition; and that whenever a suitable lift is known to exist, one can be chosen.
  Putting these together, the proposition follows.
\end{proof}

\noindent Similarly, as for ordinary categories, univalence bounds the h-level of the types of objects:

\begin{proposition}[\coqident{Core}{univalent_disp_cat_has_groupoid_obs}]
 Let \(\D\) be a univalent displayed category over \(\C\).
 Then for each \(c:\C\), the type of objects \(\dob{\D}{c}\) is a 1-type. \qed
\end{proposition}

\subsection{Structure Identity Principle} \label{sec:SIP}

Theorem~\ref{theorem:total_univalent} generalizes an early-noted consequence of univalence, the so-called \emph{structure identity principle}, as formulated by Aczel.
We recall here the version from the HoTT book; a slightly different formulation is considered in \cite{Coquand20131105}.

\begin{defiC}[{\cite[Def.~9.8.1]{HoTTbook}}] A \emph{standard notion of structure} on a category \(\C\) consists of:
\begin{enumerate}
 \item for each \(c : \C\), a type \(P(c)\); \label{item:sns-obs}
 \item for each \(c, c' : \C\) and \(\alpha : P(c)\) and \(\beta : P(c')\) and \(f : \C(c,c')\), a proposition \(H_{\alpha,\beta}(f)\);
 \item such that \(H\) is suitably closed under composition and identity; and \label{item:sns-ops}
 \item for each \(c : \C\), the preorder on \(P(c)\) defined by setting \(\alpha \leq \alpha'\) if \(H_{\alpha,\alpha'}(1_c)\) is a poset. \label{item:sns-univalent}
\end{enumerate}
\end{defiC}

Items~\ref{item:sns-obs}--\ref{item:sns-ops} can immediately be read as providing an associated displayed category over \(\C\) (\coqident{SIP}{disp_cat_from_SIP_data}), whose displayed hom-sets are propositions.
The category of \emph{\((P,H)\)-structures}, as defined in \cite{HoTTbook}, is precisely the total category of this displayed category.

With a little thought, item~\ref{item:sns-univalent} can then be seen as saying that this displayed category is univalent (\coqident{SIP}{is_univalent_disp_from_SIP_data}).
Theorem~\ref{theorem:total_univalent} then immediately implies:
\begin{corC}[{\cite[Theorem 9.8.2]{HoTTbook}}]
 Given a standard notion of structure \((P,H)\) on \(\C\), if \(\C\) is univalent, then so is the category of \((P,H)\)-structures on \(\C\). \qed
\end{corC}

\begin{example}[\coqident{Examples}{is_univalent_disp_functor_alg}]
  The displayed categories of algebras for an endofunctor or monad (Examples~\ref{ex:endofunctor-alg-displayed}, \ref{ex:monad-alg-displayed}) arise from standard notions of structure, and so are univalent.
\end{example}

\subsection{Amnestic functors}\label{subsec:amnestic}

Univalence of categories beyond posets is not typically considered explicitly in the classical setting, since when all types are sets, only a category containing no non-trivial automorphisms can be univalent.
However, the functors corresponding to univalent displayed categories can be recognised in the established (though comparatively little-used) notion of \emph{amnestic} functors.
To compare them in the univalent setting, we must clarify the classical vocabulary a little.
By saying that a morphism \(f : a \to b\) in a category \(\C\) \emph{is an identity}, we mean this in the total type of morphisms of \(\C\): that is, that there exists some \(c : \C\) such that \((a,b,f) = (c,c,1_c) : \Sigma_{x,y:\C}\mor[\C]{x}{y}\).

\begin{definition}[{cf.\ \cite[Def.~3.27(4)]{adamek-herrlich-strecker}}]
  A functor \(F : \C' \to \C\) is:
  \begin{enumerate}
  \item \emph{weakly amnestic} if for any isomorphism \(i : a \iso b\) in \(\C'\), \(i\) is an identity if and only if \(Fi\) is an identity;
  \item \emph{amnestic} if for any isomorphism \(i : a \iso b\) in \(\C'\), the map from ‘objects \(c : \C'\) such that \((a,b,i) = (c,c,1_c)\)’ to ‘objects \(c : \C\) such that \((Fa,Fb,Fi) = (c,c,1_c)\)’ is an equivalence.
  \end{enumerate}
\end{definition}

The established definition of amnestic is usually phrased as what we have called weakly amnestic.
However, in the classical setting, they are equivalent; so either may be seen as a reasonable type-theoretic reading of the classical definition:

\begin{proposition}
  If \(\C\) is a category whose type of objects is a set, then for any \(f : a \to b\) in \(\C\), the type of ‘objects \(c\) such that \((a,b,f) = (c,c,1_c)\)’ is a proposition.

  Thus if \(\C\) and \(\C'\) both have sets of objects, a functor \(F : \C' \to \C\) is amnestic if and only if it is weakly amnestic. \qed
\end{proposition}
%
%
%

We then have:

\begin{proposition}
  Let \(\C\) be any category, and \(\D\) a displayed category over \(\C\).
  Then \(\D\) is univalent exactly if \(\forget[\D] : \total{\D} \to \C\) is amnestic.
\end{proposition}

\begin{proof}
  For any map \(f : a \to b\) in \(\C\), the type of ‘objects \(c\) such that \((a,b,f) = (c,c,1_c)\)’ is equivalent to the type of ‘equalities \(e : a = b\) such that \(f = \idtoiso(e)\)’.
  For \((f, \f) : (a,\bar{a}) \to (b,\bar{b})\) in \(\total{\D}\), the analogous type is further equivalent to the type of pairs \(e : a = b\) and \(\bar{e} : \bar{a} \depeq[e] \bar{b}\) such that \(f = \idtoiso(e)\) and \(\f \depeq \idtoiso(e,\bar{e})\).
  
  Moreover, the map between these types induced by \(\forget[\D]\) is the evident projection map, so is an equivalence just if for any \(e : a = b\) such that \(f = \idtoiso(e)\), there is a unique \(\bar{e}\) such that \(\f \depeq \idtoiso(e,\bar{e})\).

  So \(\forget[\D]\) is amnestic just if this holds for every isomorphism in \(\total{\D}\).
  By the quantification over \(e\) such that \(f = \idtoiso(e)\), this is equivalent to the statement: for every \(a\), \(b\), \(e : a = b\), \(\bar{a} : \dob{\D}{a}\), \(\bar{b} : \dob{\D}{b}\), and \(\f : \diso{\bar{a}}{\idtoiso(e)}{\bar{b}}\), there is a unique \(\bar{e}\) such that \(\f \depeq \idtoiso(e,\bar{e})\).
  But this is clearly equivalent to univalence of \(\D\).
\end{proof}

\section{Conclusions} \label{sec:conclusion}

We have introduced displayed categories, and set up their basic theory, along with key examples and applications.

The applications fall into two main groups:
\begin{enumerate}
\item rephrasing classical definitions to avoid referring to equality of objects;
\item allowing categories of multi-component structures, and maps between such categories, to be constructed and reasoned about in a modular, stage-by-stage fashion.
\end{enumerate}

In this paper, we have focused more on the former---for instance, the use of displayed categories as a basis for the development of fibrations in the type-theoretic setting.

We have seen less of the latter, since it is typically tied to specific more involved applications.
However, in our own further work (for instance, on the structures considered in \cite{lmcs:4814}), we have found this at least as significant as a payoff of the present work.

Theorem~\ref{theorem:total_univalent}, giving univalence of the total category, is especially valuable.
Na\"ive approaches to proving univalence quickly become quite cumbersome even for 
categories of only moderately complex structures, such as groups.
The issue is that identities between such structures translate to a tuples of identities between the components, where the identities of later components are usually heterogeneous, involving accumulated transports along the identities between earlier components.

The displayed-category approach avoids this; one need only work `fibrewise', over each component in turn.
All the necessary wrangling of transports is dealt with once and for all in the proof of Theorem~\ref{theorem:total_univalent}.

An instance of this is the proof of univalence of the category of CwF-structures over a fixed univalent base category.
Details are beyond the scope of the present article, but it is available in the formalisation as \coqident{}{is_univalent_term_fun_structure}.
\plan{TODO: actually put that together into the total univalence!}

\paragraph*{Further work}

In the present article and formalisation, we have explored only the basic theory and applications of displayed categories.
There are many clear directions for further work:
\begin{enumerate}
\item In \cite{lmcs:4814}, we have started a project of giving careful comparisons
  between the various categorical structures used for semantics of type theory.
  We touched on this project in Section~\ref{sec:comp-cats}.
  In forthcoming work, we plan to give full comparisons between categories of such structures,
  including comprehension categories, type-categories (not necessarily split), and categories with display maps.

\item The material on creation of limits in Section~\ref{sec:limits} should be generalised to a more permissive notion of \emph{displayed limits}, to cover a broader range of examples.

\item In the formalisation (though not the article) we study displayed adjunctions and equivalences over a fixed base, and show that these induce adjunctions and equivalences between total categories and fibre categories.
  This should be generalised to displayed adjunctions/equivalences over adjunctions/equivalences in the base. 
  
\item Generally, one should be able to assemble displayed categories into a displayed bicategory over the bicategory of categories.
  Of course, this would require defining displayed bicategories, and developing the basic theory of bicategories in the type-theoretic setting.

\item Displayed categories should also be viewable as forming some 2-dimensional analogue of a comprehension category, with displayed categories being the ‘dependent types’ over a base category ‘context’.
  This would provide a new potential guiding example for the ‘directed type theory’ that various authors have started to explore in recent work.
\end{enumerate}

\paragraph*{Acknowledgements.}
We would like to thank Mike Shulman and the participants of the Stockholm Logic Seminar for very helpful feedback on the present work.

\thegrants


\bibliographystyle{amsalphaurlmod}

\bibliography{../literature}

\end{document}